\documentclass[11pt]{amsart}

\usepackage[dvips]{graphicx}
\usepackage{calc}
\usepackage{color}
\usepackage{amsmath}
\usepackage{amssymb}
\usepackage{amscd}
\usepackage{amsthm}
\usepackage{amsbsy}
\usepackage{delarray}
\usepackage{enumerate}
\usepackage[T1]{fontenc}
\usepackage{inputenc}
\usepackage{enumerate}
\usepackage{hyperref}

\usepackage{amscd}
\usepackage{amsthm}
\usepackage{array}
\usepackage{indentfirst}
\usepackage{geometry}
\usepackage{epsfig}

\usepackage{dsfont}
\usepackage{tikz}
\usetikzlibrary{matrix,arrows}
\usepackage{mathrsfs}
\usepackage[all]{xy}

\DeclareMathOperator{\im}{im}

\renewcommand{\leq}{\leqslant}
\renewcommand{\geq}{\geqslant}

\title{Spectral invariants towards a non-convex Aubry Mather theory} 
\author{Nicolas Vichery}
\date{\today}

\newtheorem{thm}{Theorem}[section]
\newtheorem{lemma}[thm]{Lemma}
\newtheorem{prop}[thm]{Proposition}
\newtheorem{coroll}[thm]{Corollary}

\theoremstyle{definition}

\newtheorem{definition}[thm]{Definition}

\newtheorem{conj}[thm]{Conjecture}

\newtheorem*{acknow}{Acknowledgement}
\newtheorem*{thmintro}{Theorem}

\newenvironment{prf}[1][]{\addvspace{8mm} \emph{Proof #1.
    ---~~}}{~~~$\Box$\bigskip}

\newlength{\espaceavantspecialthm}
\newlength{\espaceapresspecialthm}
{\vskip \espaceapresspecialthm}

\newenvironment{rem}[1][]{\refstepcounter{thm} 
\vskip \espaceavantspecialthm \noindent \textsc{Remark~\thethm #1.} }%
{\vskip \espaceapresspecialthm}

\newcommand{\R}{{\mathbb{R}}}

\newcommand{\T}{{\mathbb{T}}}
\newcommand{\Z}{{\mathbb{Z}}}

\newcommand{\cA}{{\mathcal{A}}}

\newcommand{\cG}{{\mathcal{G}}}

\newcommand{\cL}{{\mathcal{L}}}
\newcommand{\cM}{{\mathcal{M}}}

\newcommand{\fc}{{:\ }}

\newcommand{\tb}{\textbf}

\begin{document}

\maketitle

\begin{abstract} 
Aubry-Mather is traditionally concerned with Tonelli Hamiltonian (convex and super-linear). In \cite{Vi,MVZ}, Mather's $\alpha$ function is recovered from the homogenization of symplectic capacities. This allows the authors to extend the Mather functional to non convex cases. This article shows that the relation between invariant measures and the subdifferential of Mather's functional (which is the foundational statement  of Mather) is preserved in the non convex case.

We give applications in the context of the classical KAM theory to the existence of invariant measures with large rotation vector after the possible disappearance of some KAM tori.
\end{abstract}

\section{Introduction}

In all this paper, we assume $M$ to be  a compact finite dimensional connected manifold endowed with an auxiliary Riemannian metric $g$. We consider the cotangent bundle $T^*M$ equipped with the Liouville form $\theta$.

\begin{definition}
A Hamiltonian $H$ is Tonelli if :

\begin{itemize}
\item $H$ is at least $C^2$;
 \item  $H$ is superlinear, meaning that $\lim\limits_{|p|\to \infty}\frac{H(q,p)}{|p|} = \infty$;
 \item  $H$ is strictly convex in the fibers.
\end{itemize}

\end{definition}

Aubry-Mather theory deals with a more general situation than KAM theory. This last one provide the existence of invariant tori after small perturbation of integrable system with large degree of smoothness. However, invariants tori are destroyed far from the integrable case. Aubry-Mather theory tries to give information about invariant sets and measures in this context.

Traditionally, Tonelli Hamiltonians are the only ones treated in Aubry-Mather theory. In this case the study of the Hamiltonian dynamics can be bring back to a minimization problem thanks to the duality of Legendre transform. These hypothesis are also related to the specific form of the Lagrangians considered in physics.

For non-convex Hamiltonians, it is classical in symplectic geometry to introduce action integrals and study variational problems. Generally, these integrals do not present any minimum or maximum. That is the main reason why one have to consider minmax technics.

Given a Tonelli Hamiltonian $H$,  Mather associates a functional $\alpha_H \fc H^1(M)\to \R$ defined by:

\[\alpha_H([\eta]):=-\min \limits_{\mu\in \mathcal{M}_{inv}}(\int_{TM} L-\eta d\mu )\]

\noindent where $L$ is the Legendre dual of $H$ and $\mathcal{M}_{inv}$ is the set of  probability measure on $TM$ invariant by the Euler-Lagrange flow. This functional seems to encapsulate many information on the dynamics \cite{Mather_action_minimizing}.

Attempt to generalize Aubry-Mather theory to the non convex case exists in the case of $T^*\mathbb{T}^n$. Cagnetti, Gomes and Tran in \cite{Cagn} study the existence of invariant measures for non-convex autonomous Hamiltonians generalizing Mather measures. An extension of Mather $\alpha$ functional for non convex function has been given by the process of symplectic homogenization in \cite{Vi} and generalized for all cotangent bundle in \cite{MVZ}.

Mather fundamental results are the following (restated in Hamiltonian form):

\begin{thm}\cite{Mather_action_minimizing}
 Let $H$ a Tonelli Hamiltonian. For all $\xi\in\alpha_H(c)$, there exists an invariant measure $m$ on $T^*M$ with rotation vector $\xi$ and action $\alpha(c)-\langle \xi,c\rangle$.
\end{thm}

He deduces using the  convexity of $\alpha$ that all rotation vector can be reach.

We generalize this statement in the present article to the non-convex case for all cotangent bundle Thm. \ref{main_thm} and Thm. \ref{nonautonomous}.

\begin{thmintro}[\ref{nonautonomous}]
Let $H\in C^2(T^*M\times\mathbb{S}^1)$ with geometrically bounded flow. Then, for all $\eta\in \partial \alpha_H(\hat\lambda)$, the exists $m$ a $\phi^1_H$ invariant measure such that $\cA(m)=\alpha_H(\hat\lambda)-\langle\eta, \hat\lambda\rangle$ and $\rho(m)=\eta$.
\end{thmintro}

 We must mention that in an article in preparation, Viterbo gives another proof of this in the case of the torus \cite{Vitnonconvex}. He also obtains invariant measures from Lagrangian intersection problem but with a different class of Lagrangians of the torus than those of the present work. It seems difficult for other class of manifolds to generalize his construction.

A second fundamental theorem of Mather says that these invariant measures are supported on a Lipschitz graph. It is expected that in the non convex case, this cannot occur. Nevertheless, we can ask if the measure constructed in the present article can be supported on a $C^0$ Lagrangian into the cotangent bundle in the sense of  \cite{Humi}.

The generalization of the methods of Theorem \ref{nonautonomous} allows us to formulate statement about the localization of invariant measures. This is the content of Theorem \ref{local}.

\begin{acknow}
I would like to warmly thank  Vincent Humili\`ere for useful comments on a preliminary version of this work. I am also grateful to Claude Viterbo for asking me if it was possible to generalize Mather's theorem about existence of invariant measures to the non-convex case.

The research leading to these results has received fundings from the European Community's Seventh Framework Programme ([FP7/2007-2013][FP7/2007-2011]) under grant agreement $n^\circ$[258204] and from the ANR Weak KAM Beyond Hamilton-Jacobi (ANR-12-BS01-0020).
\end{acknow}

\section{Review of symplectic homogenization}
\subsection{Spectral invariants}

Given two Lagrangians $L_1$ and $L_2$ isotopic to the zero section, we can associate to them different (cohomology)groups using Floer theory, generating functions or microlocal analysis of sheaves.

All this group are known to be isomorphic. We refer to \cite{Milinkovic_Oh_gf_versus_action} and \cite{Viterbo_functor} for the isomorphism between Floer group and generating function homology and to \cite{HDC} for the construction of sheaves that generalizes the framework of generating functions.

Moreover, these groups are naturally filtered by $\R$ called ``action'' variable. Therefore, it is natural to consider the values of the action where some elements of the group disappear. That is the notion of spectral invariant defined by Viterbo in the context of generating function and Schwartz and Oh for Floer homology. They can be seen as "homologically visible" critical values. They do not depends on the technic we use since the previous isomorphisms respect the $\R$-filtration.

For simplicity, we will use generating function technics.

\begin{rem}
In order to treat other symplectic manifolds, we should have used Floer homology, which has been so far the only technical way to deal with general symplectic manifold. Nevertheless, to follow the same path than the present paper, we would need in the largest class of generality a well developed version of family Floer homology.
\end{rem}

\subsubsection{Definition}\label{section_definition_Lagr_sp_invts_gf}

A generating function quadratic at infinity, or gfqi is a function $S \fc M(q) \times E(\xi) \to \R$, with  $E$ a vector space of finite dimension, such that $\|\partial_\xi S - \partial_\xi B\|_{C^0}$ is bounded and $B \fc E \to \R$ is a non degenerate quadratic form. We denote $E = E^+ \oplus E^-$ the decomposition in positive and negative eigenspaces of $B$.

We consider the relative homology $H_*(\{S<a\},\{S<b\})$. For $a$ large and $b$ small enough this group is independant of $(a,b)$ and is canonically isomorphic to $H_*(M) \otimes H_*(E^-,E^--0) \simeq H_{*+d}(M)$, with $d = \dim E^-$ where the last isomorphism ("Thom's isomorphism" ) is given by tensorising  with the generator of $H_d(E^-,E^--0) \simeq \Z_2$. We denote this group by $H_*(S:M)$. 

There exists a natural morphism induced by the inclusion $i^a\fc H_*(\{S < a\},\{S < b\}) \to H_*(S:M)$, b small enough. To every $\alpha\neq 0 \in H_*(M)$ we can associate the spectral invariant:
\[l(\alpha,S) = \inf \{a\,|\,\alpha\in\im i^a\}\ .\]

These invariants are defined for Lagrangian submanifolds $T^*M$. Indeed a gfqi gives  rise to Lagrangian immersion \cite{Vit_gfqi} as follows.

\begin{definition}
Let $(q,\xi)$ be the coordinates on $M\times\mathbb{R}^k$ and $S\fc M\times\mathbb{R}^k \rightarrow \mathbb{R}$ a smooth function, $S$ is a generating function for $L$ if:

\begin{enumerate}
 \item The map $(q,\xi)\mapsto \partial_\xi S(q,\xi)$ has 0 as regular value.
 \item The manifold $\Sigma_S=\{(q,\xi) | \partial_\xi S(q,\xi)=0\}$ is compact in $M\times\mathbb{R}^k$ and
 \begin{align*}
            i_S: \Sigma_S & \rightarrow T^*M \\
           (q,\xi) & \mapsto(q,\partial_q S(q,\xi))
 \end{align*} 
 has $i_S(\Sigma_S)=L$ as image.
\end{enumerate}
\end{definition}

By a Theorem of Viterbo and Th\'eret, for all Lagrangian hamiltonialy isotopic to the zero section there exists a ``essentialy unique'' generating function. The existence was already proved by Chaperon Laudenbach and Sikorav.

\begin{thm}[Laudenbach-Sikorav]
  If $L$ is a Lagrangian generated by a gfqi and $\phi$ is a Hamiltonian diffeomorphim with compact support then $\phi(L)$ can be generated by a gfqi.
\end{thm}

The ``essential uniqueness '' is only true according to an equivalence relation which is up to fibered diffeomorphism and addition of a constant \cite{Vit_gfqi}.

 We can associate spectral invariants to Lagrangian submanifolds hamiltonialy isotopic to the zero section because of the following remark. According to \cite{Milinkovic_Oh_gf_versus_action}, it is possible to normalize generating function so that: 
 
 \[S(x,\xi)=\int_\gamma -\theta+H_t\]
 
 \noindent with $\gamma(1)=(x,\frac {\partial S}{\partial x}(x,\xi))$, $\dot \gamma=X_H$ and $\frac{\partial S}{\partial \xi}(x,\xi)=0$.

\begin{rem}
Any non degenerate quadratic form is a gfqi for the zero section.
\end{rem}

%
%

\subsection{Review of symplectic homegenization} \label{mainresult}

It has been first defined by Viterbo on the cotangent bundle of $\T^n$,  by some process close to homogenization in the field of partial differential equations \cite{LPV}. Then the construction has been generalized to every cotangent bundle \cite{MVZ}. In this context, the use of the Lagrangian Floer homology has given insights to prove symplectic properties as well as relation to Mather's $\alpha$ functional.
Let us denote by $\cG$, the group of Hamiltonian diffeomorphisms with compact support.

\subsubsection{Construction} 

\begin{definition}\cite{MVZ}
Let $a=[\beta]\in H^1(M)$. then we define $\mu_a:\cG\rightarrow\R$ by:
\[\mu_0(\phi):=\lim\limits_{k\rightarrow\infty}\frac{l_+(\phi^k)}{k}\]
If $\phi=\phi_H$, and $K_\alpha=H(x,p-\beta(x))$ then :
\[\mu_a(\phi):=\mu_0(\phi_K)\]
\end{definition}

This definition does not depend on the choice of a representative for $\beta$ \cite{MVZ}.

\subsubsection{Property}

In this section we present the main property of that functional $\mu_a$ :

\begin{thm}\label{MVZmain}\cite{MVZ} Let $M$ a connected manifold. Then for  all $a \in H^1(M;\R)$, the function $\mu_a \fc \cG \to \R$ satisfies :

\begin{enumerate}
\item $\mu_a(\phi^k)=k\mu_a(\phi)$ for $k \geq 0$ an integer;
\item $\mu_a$ is invariant by conjugaison in $\cG$;
\item If $\phi,\psi\in\cG$ are generated by Hamiltonians $H,G$. Then
\[\int_0^1\min(H_t-G_t)\,dt \leq \mu_a(\phi)-\mu_a(\psi) \leq \int_0^1\max(H_t-G_t)\,dt\,;\]
in particular $\mu_a$ is  Lipschitz according to the Hofer metric.
\item The restriction of $\mu_a$ to the subgroup of $\mathcal{G}$  of diffeomorphism generated by the Hamiltonians supported on a displaceable open set $U$ is zero;
\end{enumerate}

\end{thm}

\begin{rem}
These definitions and properties extend to all complete flows \cite{MVZ}. 
\end{rem}

\begin{thm}\cite{MVZ}
Let $H$ be a Tonelli Hamiltonian. Then $\alpha_H(a)=\mu_a(\phi_H)$.
\end{thm}

The previous statement allows us to see $\mu$ as a generalization of Mather's $\alpha$ functional for non convex Hamiltonian. Moreover, it recovers the symplectic invariance property (as in \cite{Bernard_sympl_aspects}) and extend it. Indeed, Bernard proved symplectic invariance of the $\alpha$ functional for symplectomorphisms that preserve the Tonelli property. Others statement of the classical $\alpha$ functional can be recovered according to Theorem \ref{MVZmain}.

\subsubsection{Reformulation}

We want to consider all parameters in $H^1(M)$ in a single generating function. We choose a basis of $H^1(M;\R)$ denoted by $\{\beta_i\}$ and consider the following Hamiltonian $\tilde H\fc T^*M\times T^*\R^n$, $n=dim_\R(H^1(M;\R))$:

 \[\tilde H(x,p,\lambda,\lambda^*):=\chi(\lambda) K_{\Sigma\lambda_i.\beta_i(x)}(x,p)\]
\noindent with $\chi$ a compact supported function with value $1$ on a large compact set $A\subset \R^n$. Indeed, we are interested in properties that are local with respect to $\lambda$ that we will suppose belonging to $A$. The compact cut-off is here only to ensure the existence of some generating function.

\begin{rem}
To make the notation as light as possible, we introduce the notation $\hat{\lambda}$ instead of $\sum \lambda_i \beta_i$.
\end{rem}

We study the dynamics induced by this Hamiltonian. The image of $M$ by its time $k$ diffeomorphism has a generating function $\tilde S_k\fc M\times \R^n\times\R^k\to \R$ because of the compactness of the support.
The dynamics lets $\lambda$ constant along the evolution. So, critical orbits going from the zero section  to the conormal of  $M\times \lambda_0$ for some $\lambda_0\in A$ satisfy $\lambda=\lambda_0$. They are critical orbits for the Hamiltonian $K_{\lambda_0}$.

\begin{prop}
Let \[ f_k(\lambda):=\frac{l([M]\otimes [1_\lambda], \tilde S_k)}{k} \ .\]
Then, the sequence $f_k$ converges uniformly on all compact subset and \[\mu_{\hat \lambda} (\phi_H):=\lim\limits_{k\to\infty}f_k(\lambda) \ .\]
\end{prop}

\begin{proof}
First, we prove that $f_k(\lambda)=\frac{l_+(\phi^k_{K_{\hat\lambda}})}{k}$.


We look at the Hamilton equations associated to $\tilde H$. Because of the independence according to $\lambda^*$, $\lambda(t)=\lambda(0)$ and the dynamics restricted to $\lambda=\lambda(0)$ is the same as the one defined by $K_{\Sigma\lambda_i a_i}$ for $\lambda \in A$ fixed. There is a one to one correspondence between the orbits of $\tilde H$ going from $M\times\R^n$ to the conormal of $M\times \{\lambda\}$ with the orbits of $K_{\hat\lambda}$ (fixed $\lambda$) going from and to the zero section. Moreover they have the same action, since $d\lambda$ vanishes.

Moreover, $\tilde S(x,\lambda,\xi)=:S_\lambda(x,\xi)$ with fixed $\lambda \in A$ is a gfqi of $\phi_{K_{\hat \lambda}}(M)$ with the correct normalization.

 
 More precisely, let's denote the orbit of $q\in M$ by $\gamma_{\phi(q)}(t):=\phi^t(q)$. Then, for all $\lambda \in A$ the normalization of $\tilde S$ gives :
 
 \[\int_{\gamma_{\phi(g)}} -\theta + K_{\hat\lambda}=\int_{(\gamma_{\phi(g)},\lambda)} -(\theta +\lambda^*d\lambda )+ \tilde H = \tilde S(x,\lambda,\xi)\]
 
\noindent with $(x,\lambda,\xi)$ satisfying $\frac{\partial S}{\partial \xi}(x,\lambda,\xi)=0$ and $\gamma_{\phi_g}(1)=(x,\frac{\partial S}{\partial x})$.
 By lemma 2.16 in \cite{MVZ} which is proved in \cite{Milinkovic_Oh_gf_versus_action}, $S_\lambda(x,\xi)$ with fixed $\lambda \in A$ is a gfqi of $\phi_{K_{\hat \lambda}}(M)$.

Applying again this lemma, we get: \[l_+(\phi_{K_{\Sigma\lambda_i a_i}})=l_+(S_\lambda)=l([M]\otimes 1_\lambda,S) \ .\]

Let us now prove the uniform convergence of the sequence $f_k$.
The previous equality shows that $f_k$ converges point-wise. The uniform convergence then follows from the fact that the function $f_k$ are equi-Lipschitz on compact sets.
Indeed, we use the following bound on spectral invariants due to Oh for all Hamiltonian $H$ and $G$:

\[l_+(\phi_H)-l_+(\phi_G)\leq \int_0^ 1 \max(H_t-K_t) dt \ .\]

\noindent In our situation the computation is similar to \cite{MVZ} proof of (ix) of the main theorem.
It follows that,

\[l_+(\phi_H^k)-l_+(\phi^k _G)\leq k\int_0^ 1 ||H_s-G_s||ds \ .\]

\noindent Hence,

\[f_k(a)-f_k(b)\leq \int_0^1|dH_t|(a-b)<C|a-b| \]

\noindent where for all 1-form $\chi$, $|dH_t|(\chi):=\max_{(q,p)\in T^*M} |\langle d_{(q,p)}H_t|_{T_{(q,p)}^{vert}T^*M},\chi\rangle |$ and $C$ only depend on the radius of a large ball containing $a$ and $b$.

Thus, $f_k$ converges uniformly.
\end{proof}

\section{Subdifferential}

As in the classical Aubry Mather theory,  we want to link the subdifferential of Mather's functional with the existence of invariant measures with action  and rotation vector prescribed. There exists in the literature a wide number of subdifferentials. In the classical theory the functional is convex, so it is convenient to use the convex subdifferential (Rockefellar-Moreau defined through Hahn-Banach theorem). This is the dual convex cone to the epigraph of the functional.
 
In our framework (non convexity, many different definitions exist and often differ. Nevertheless, all subdifferentials satisfy a certain family of Axioms  \cite{Ioffe}. It constitutes the field of Non-smooth analysis for which we refer to \cite{Clarke,Ioffe}.

However, for our purpose, we want to use a subdifferential that behaves well with respect to to $C^0$ convergence of functions, which is related to symplectic geometry of the cotangent bundle and as big as possible.


\begin{rem}
In order to be as general as possible, we give two requirements for a subdifferential $\partial$ to satisfy the future theorems.
  
  \begin{enumerate}[I]
   \item  For all $f\fc \R^n \to \R$ Lipschitz, $\partial f \subset \partial_c f$;
   \item  Let $X$ a manifold and $f_n \fc X\to \R$ a sequence of $C^0$ functions that $C^0$ converge to $f$. Then,

\[\partial f\subset \limsup \partial f_n\]

\noindent with $\limsup \partial f_n:= \bigcap\limits_{k>0}\overline{\bigcup\limits_{n>k} \partial {f_n}}$.

  \end{enumerate}

  Because of the locality property of subdifferentials \cite{Ioffe} (independance with respect to the complement of small neighborhoods), the $C^0$ convergence assumption can be replaced by $C^0$ convergence on every compact.

From now, all subdifferentials except the Clarke subdifferential will satisfy property (I) and (II).  They will be denoted by $\partial$. The main example being the homological subdifferential.
  
 \end{rem}

As far as the author knows, only two such subdifferentials exist: the approximate (or $G-$) subdifferential $\partial_a$ \cite{jourani} and the homological subdifferential $\partial$ \cite{HDC}.

\begin{definition}
 Let $A\subset T^*M$. We denote by $co(A)$ the convex hull of A in each fiber.
\end{definition}

\begin{definition}
Let $f:X\to \R$ be a lower semi-continuous function, $\rho:T^*(X\times\R) \to T^*X$, $\rho(x,t;\xi,\tau):=(x,\frac{\xi}{\tau})$. The subset $\partial f\subset T^*M$ defined by : \[\partial f:=\rho(\dot{SS}(\R_{\{f(x)\leq t\}}))\] 

\noindent with $\R_{f(x)\leq t}$ the constant sheaf on the epigraph of $f$ and $\dot{SS}$ the complement of the zero section in the singular support (see \cite{KS,HDC}).
\end{definition}

\begin{definition}\cite{Clarke}
 Let $f:X\to \R$ be a Lipschitz function. Then the Clarke subdifferential is \[\partial_c f(\lambda):=co(\lim\limits_{n\to \infty}(df(\lambda_n)),\lambda_n\to\lambda \ for \ almost\ every \lambda_n).\]
\end{definition} 

 Almost all generalization of subdifferential agree with this one in the convex case. But here, the homogenized Hamiltonian is no longer convex and is Lipschitzian (even continuous if we start with $H$ continuous).

It has  been proved in \cite{HDC} that for any Lipschitz function $f$ :

\[\partial_a f \subset \partial f \subset \partial_c f=co(\partial_a f)\ .\] 

Strict inclusions are known to occur but this formula is also sufficient to prove the non emptiness of the subdifferentials considered here. Indeed, $\partial_a$ is known to be non-empty at every point \cite{Ioffe}.

Moreover, $\partial f$ fits very well in the framework of \cite{Tam}, \cite{GKS} about non-displaceability of Lagrangian in cotangent bundle using sheaf theoretical methods. Note that an unified approach of Mather's theory through sheaf theory would give new insights.

Let us mention two of the main properties of the homological subdifferential.

\begin{prop}
Let $X$ be a manifold and $f_n \fc X\to \R$ a sequence of $C^0$ functions that $C^0$ converge to $f$. Then,

\[\partial f\subset \limsup \partial f_n\ .\]

\end{prop}

\begin{prop}\label{inclusion}
Let $S\fc  \R^n\times X\times \R^q\to  \R$ such that $S_\lambda$ is a gfqi. Then the normalized action selector $f$ satisfies :

\[\partial  f(\lambda)\subset co(\left\{ d_\lambda  S_(\lambda,x,\xi)\ |\  (x,\xi) \in C(\lambda)\right\}) \ .\]

where the set $C(\lambda)$ is defined by \[C(\lambda):=\{(x,\xi)|d_{(x,\xi)}  S(\lambda,x,\xi)=0, S(\lambda,x,\xi)=f(\lambda)\} \ . \]
\end{prop}

\begin{proof}
We will show the bound for the Clarke subdifferential which will imply it for $\partial$.
It can be proved by methods extracted from \cite{Wei} lemma 2.27. In this article, the author deals in a larger class of generality, $S$ being non-smooth. For the sake of completeness, we rewrite the proof in our context.

Here is a condensed formula for spectral invariant we can deduce from the definition:

\[f(\lambda):=l(1_\lambda\otimes [M],S)=\inf\limits_{[\sigma]=1_\lambda\otimes [M]} \max\limits_{(\lambda,x,\xi)\in \sigma} S(\lambda,x,\xi)\ .\]

According to lemma 2.25 of \cite{Wei} :

\[\forall \delta>0, \exists \epsilon>0, s.t. f(\lambda)=\inf\limits_{\stackrel{[\sigma]=1_\lambda\otimes [M]}{\sigma\in\Sigma_\epsilon}} \max\limits_{(\lambda,x,\xi)\in \sigma\cap C^\delta(\lambda)} S(\lambda,x,\xi) \]

\noindent where $\Sigma_\epsilon:=\left\{\sigma\ |\ \max\limits_{(\lambda,x,\xi)\in\sigma} S(\lambda,x,\xi)\leq f(\lambda)+\epsilon \right\}$ and $C^\delta(\lambda)$ is a $\delta$ neighborhood of $C(\lambda)$.

The proof of this lemma is even easier in our context because functionals $S_\lambda(x,\xi)$ are smooth and satisfy the Palais-Smale property. We can deduce it from a classical deformation lemma for sublevels in the complement of $\delta$ neighborhood of critical points.

First, we suppose that $f$ is differentiable at $\lambda_0$. Let us consider a sufficiently small $s\in (0,1)$ such that for all $\lambda\in B_s(\lambda_0)$ then:

\[|S(\lambda,.)-S(\lambda_0,.)|\leq \frac{\epsilon}{4} \ .\]

\noindent This is possible because of the continuity of $S$ and because $S_\lambda$ is a generating function quadratic at infinity that can be chosen equal to a fixed quadratic form at infinity.

Let us define $u\in\R^n$, $v\leq 0$ with $\lambda_v:=\lambda_0+vu \in B_s(\lambda_0)$ and $v^2<\frac{\epsilon}{4}$. It is then easy to choose a cycle $\sigma_v$ in the right homology class such that:

\[\max\limits_{\sigma_v}S(\lambda_v,x,\xi)\leq f(\lambda_v)+v^2 \ .\]

It follows that:

\[\max\limits_{\sigma_v} S(\lambda_0,x,\xi)\leq \max\limits_{\sigma_v} S(\lambda_v,x,\xi)+\frac{\epsilon}{4}\leq f(\lambda_v)+\frac{\epsilon}{2}\leq f(\lambda_0)+\frac{3\epsilon}{4} \ .\]

Applying the lemma 2.25 from \cite{Wei}, we get:

\[f(\lambda_0)\leq \max\limits_{\sigma_v\cap C^\delta(\lambda_0)} S(\lambda_0,x,\xi)=S(\lambda_0,x_v,\xi_v) \]

with $(x_v,\xi_v)\in C^\delta(\lambda_0)\cap \sigma_v$.

Hence by a mean value argument, there exists $\lambda'_v$ such that,

\[v^{-1}(f(\lambda_v)-f(\lambda_0))\leq v^{-1}(S(\lambda_v,x_v,\xi_v)-S(\lambda_0,x_v,\xi_v))-v\leq d_\lambda S(\lambda'_v,x_v,\xi_v)-v\ .\]

Let us consider the upper limit of both sides. Then,

\[\langle df(\lambda_0),u \rangle\leq \max\limits_{(x,\xi)\in C(\lambda_0)} \langle d_\lambda S(\lambda_0,x,\xi),u \rangle,\ for\  u\in \R^n \ .\]

It implies by definition of subdifferential of convex functions that $df(\lambda_0)\in \partial g(0)$ with $g(y):=\max\limits_{(x,\xi)\in C(\lambda_0)}(\langle d_\lambda S(\lambda_0,x,\xi),y\rangle$. An easy computation shows that:

\[df(\lambda_0)\in co\left\{ d_\lambda S(\lambda_0,x,\xi),(x,\xi)\in C(\lambda_0) \right\}\ .\]

It follows in the general case (where f is  possibly not differentiable at $\lambda_0$) that:

\[\partial_c f(\lambda_0)=co\{ \lim\limits_{\lambda\to\lambda_0} df(\lambda)\}=co \left\{ \lim\limits_{\lambda\to\lambda_0} d_\lambda S(\lambda,x,\xi), (x,\xi)\in C(\lambda) \right\}\]

\[co \left\{ \lim\limits_{\lambda\to\lambda_0} d_\lambda S(\lambda_0,x,\xi), (x,\xi)\in C(\lambda_0) \right\}\ .\]

%
%
%
%
%
%
%
%
We have to show that $\{d_\lambda S(\lambda,x,\xi)| (x,\xi) \in C(\lambda))\}$ is closed which is true by the smoothness of $S$ and the compactness of $C(\lambda)$.
 \end{proof}

\section{Invariant measures}

Usually, Aubry Mather theory deals with measures supported on the tangent bundle. Here, we are interested in the symplectic cotangent bundle and so have to formulate statements in this framework. We give a couple of definitions and propositions about measures on symplectic manifolds and more precisely on cotangent bundles. 

\begin{definition}
Let $m$ be measure on $T^*M$ and $H\in C_c^2(T^*M\times \mathbb{S}^1)$. Then, $m$ is said to be $\phi_H$ invariant if :

 \[ {\phi_H}_\sharp m=m\ .\]

\noindent where ${\phi_H}_\sharp m$ denotes the push-forward of the measure $m$ by $\phi_H$.

\end{definition}

\begin{prop}
Let $m$ be a measure on $T^*M$ supported on a compact set and $H$ an autonomous Hamiltonian. Then $m$  is  $\phi_H^t$  invariant for all $t$ if and only if,

\[ \forall f\in C_c^2(T^*M), \int\{H,f\}dm=0 \ .\]
\end{prop}

\begin{prf}
 Suppose that $m$ is $\phi_H^t$ invariant. Then, for all $(f,t)\in C_c^2(T^*M)\times\R$, 
 \[\int f(\phi_H^t(z))dm(z)=\int f(z)dm(z) \ .\]
Taking derivative with respect to $t$, we get:
 \[\int df(X_H) dm=0\ .\]
Conversely, suppose that \[ \forall f\in C_c^2(T^*M), \int\{H,f\}dm=0 \ .\]
 Then, \[\frac{d \int f (\phi_H^t(z))dm(z)} {dt}|_{t=0}=\int \{f,H\}dm=0 \]
 \[ \int f(\phi_H^t(z)dm(z)=\int f(\phi_H^0(z)dm(z)=\int f dm\ .\]
\end{prf}

\begin{prop}
Let $m$ be a $\phi_H$ invariant measure. Then, $m$ is closed (in reference to the same notion for distribution), i.e. $\forall f\in C^2(M)$,
\[\int \{f\circ\pi,H \}dm=0 \ .\] 
\end{prop}

\begin{prf}
The proof is an adaptation of the previous one. 
\end{prf}

\begin{definition}
Let $m$ be a closed measure on $T^*M$ and $H\in C^2(T^*M)$. Then we defined the rotation vector of $m$ to be $\rho(m): H^1(M;\R)\to \R$. The map is explicitly defined by:

\[ [a] \mapsto \int \langle a,X_H \rangle dm \ .\]

\end{definition}


\begin{definition}
Let $m$ be an $H$ invariant measure measure then its action is \[\cA_H(m):=-\int \left\langle \theta, X_H\right\rangle dm + \int H dm \ .\]
\end{definition}

\section{Existence of invariant measures with prescribed rotation vector}

Strategy to produce invariant measures will be to consider measures supported on Hamiltonian chords for growing times and average it.
More precisely, let us consider the following set:

\[\Gamma_{k,\lambda}:=\left\{\gamma \fc [0,k]\to T^*M\ |\ \gamma(u)\in graph (\hat \lambda), u=0\ or \ k ,\right. \] \[\left.\dot{\gamma}=X_{H}(\gamma), \int_\gamma -\theta+H= -\int_\gamma \pi^*\hat\lambda+f_k(\lambda) \right\} \ .\]

This is the set of Hamiltonian chords starting and ending on $graph(\hat \lambda)$ with action $-\int_\gamma \pi^*\hat\lambda+f_k(\lambda)$.
\begin{rem}
 We would like to stress the attention of the reader on the addition of a corrective term $ -\int_\gamma \pi^*\hat \lambda$ which appears because of the definition of $K$. Indeed, we need to bring back the critical path on $graph(\hat\lambda)$ by translation by $\hat \lambda$.
\end{rem}

To this set of Hamiltonian paths, we associate the set of measures:

\[\cM_{k,\lambda}:=\{ m_a=\frac{\gamma_\sharp \cL_{[0,k]}}{k}|\gamma \in \Gamma_{k,\lambda}\}\ , \]

\noindent where $\cL_{[0,k]}$ denotes the Lebesgue measure on $[0,k]$.

\begin{rem}
We must mention that these measures are not invariant measures as an easy calculation can show it.
Nevertheless, if it exists the limits for $k$ going to $\infty$ are invariant measures.
\end{rem}

We prove a general statement about construction of invariant measure by perturbative methods.

\begin{lemma}\label{adiabatic}
 Let $\phi^t_\lambda$ be a family of flows on  a compact manifold $Y$ parametrized smoothly by $\lambda$  and consider $\lambda_k\to 0$ and $x_k \in Y$ such that the family $\{\phi^t_{\lambda_k}(x_k)\}_k$ remains in a given compact domain. We define the sequence of probability measures $\nu_k:=\frac{1}{k}\phi^\bullet_{\lambda_k}(x_k)_\sharp \mathcal{L}_{[0,k]}$. Then there exists a subsequence of $\nu_k$ which converges weakly to a  probability measure $\nu$, $\phi^t_0$ invariant for all $t$.
\end{lemma}

\begin{proof}
By the Alaoglu theorem, we prove the existence of a subsequence converging to $\nu$. To keep notation as light as possible we rename the subsequence by $\nu_k$. We want to show that $\nu$ is $\phi^t_0$ invariant for all $t$. Let $G\in C_c^0(Y)$ and $T>0$ fixed, we get :

\[|\int G\circ\phi^T_0 d\nu -\int G d\nu| \]
\[\leq |\int G\circ\phi^T_0 d\nu -\int G\circ\phi^T_{0} d\nu_k|+|\int G\circ\phi^T_{0} d\nu_k-\int G\circ\phi^T_{\lambda_k} d\nu_k|\] \[+|\int G\circ\phi^T_{\lambda_k} d\nu_k-\int G d\nu_k|+|\int G d\nu_k - \int G d\nu|\]

The first and fourth terms go to zero by weak convergence of measures. 
To bound the second integral, we use the fact that $G\circ\phi^T_0- G\circ\phi^T_{\lambda_k}$ goes to $0$ uniformly.
To bound the third term, we look at the following estimation of length $t$ extremity of th integral for large $k$:

\[ |\int G\circ\phi^T_{\lambda_k} d\nu_k-\int G d\nu_k| \leq \frac{2T}{k}||G||_{C^0}\ .\]

\end{proof}

\begin{prop}\label{convergence}
Let $m$ be the limit of a subsequence of measures in $\cM_{k,a_k}$ and $a_k\to a$.
Then $m$ is a $\phi^t_H$ invariant measure for all $t$. Moreover, the action and rotation of the measure $m$ is the limit of action of the paths. It extends to the convex hull of $\cM_{k,a_k}$.
\end{prop}

\begin{proof}
%
%
We will use \ref{adiabatic}. Indeed, the Hamiltonian $K_{\hat\lambda}(x,p)$ is at least $C^2$ in the $\lambda$ variable. Using a version of Cauchy-Lipschitz with parameters, we obtain that the family of flows parametrized by $\lambda$ depends continuously on $\lambda$.
 The action and rotation vectors are then computed as limits on Hamiltonian path in $\Gamma_{k,a_k}$. In order to extend this result to the convex hull of $\cM_{k,a_k}$, it is sufficient to notice that the action and the rotation vector depend linearly on the measure.
\end{proof}

\begin{definition}
A Hamiltonian $H$ is said to have geometrically bounded flow if every graph of closed $1$-form evolves in a compact domain.
\end{definition}

\begin{thm}\label{main_thm}
Let $H$ be an autonomous Hamiltonian with geometrically bounded flow. Then, for all $\eta\in \partial \alpha_H(\lambda)$, there exists $m$ a measure which is invariant by $\phi_H^t$ for all $t$, such that $\cA(m)=\alpha_H(\lambda)-\langle \eta, \hat\lambda\rangle$ and $\rho(m)=\eta$.
\end{thm}

\begin{proof}
 Suppose $\eta\in \partial\alpha_H(\lambda)$. Then, $f_k$ converges uniformly to $\alpha_H$ on all compact subsets. By hypothesis (II), there exists $\lambda_{\psi(k)}\to \lambda$ , $\psi$ an extraction of $\mathbb{N}$, such that $\eta_{\psi(k)}\in\partial f_{\psi(k)}(\lambda_{\psi(k)})\to\eta$.
Without lost of generality, we will consider $\psi(k)=k$. According to hypothesis (I) on the subdifferential and proposition \ref{inclusion} , for each $\eta_k\in\partial f_k(\lambda_k)$, $\exists r_k\in\R^l$,

 \[\eta_{k}=\frac{1}{k}\sum\limits_{j\geq 0} r_{k,j} d_\lambda \tilde S_{k}(x_{k,j},\lambda_{k},\xi_{k,j})\]
 
 \noindent with $\sum\limits_j r_{k,j}=1$ and $(x_{k,j},\xi_{k,j})\in C_{k}(\lambda_{k})$.

 But,the "rotation vector" is obtained by  $d_\lambda \tilde S(x_{k,j},\lambda_{k},\xi_{k,j})=\lambda_{k,j}^*(1)=\left\{\int_{\gamma_{k,j}} \pi^*\hat\lambda_k (X_H) \right\}_{i=1}^n $ for $\gamma_{k,j}$ such that $ \int_\gamma - \theta +H +\int_\gamma \pi^*\beta_i =\tilde S_{k}(x_{k,j},\lambda_{k,j},\xi_{k,j})$.
 
 We consider the associated measure $m_{k,j}$ to $\gamma_{k,j}$. We get according to Proposition \ref{convergence} a measure $m$ with $\rho(m)=\eta$ and the action $\cA(m)=-\rho(m)(\hat\lambda)+\alpha_H(m)$.
 \end{proof}

\begin{rem}
This theorem is valid for all subdifferential satisfying axioms (I) and (II).
\end{rem}

\begin{coroll}\label{maincor}
Let $H$ be an Hamiltonian with geometrically bounded flow. Then, for all $\eta\in \partial_c \alpha(\hat\lambda)$, there exists $m$ a measure which is invariant by $\phi_H^t$ for all $t$ such that $\cA(m)=\alpha_H(\hat\lambda)-\langle\eta,\hat\lambda\rangle$ and $\rho(m)=\eta$
\end{coroll}

\begin{proof}
Let $\eta\in \partial_c \alpha_H(\hat\lambda)$. Then, according to axiom (I) for subdifferential, $\eta$ is a barycentric combination of elements in $\partial \alpha_H(\hat\lambda)$. We use Theorem \ref{main_thm} to these elements. By linearity of the action and of the rotation number, the proof follows. 
\end{proof}

\begin{rem}
The Clarke subdifferential is larger than the homological one, nevertheless, the measure in this case is obtained as a (maybe) non trivial barycenter of two invariant measures constructed with the homological one.
The measures found according to the Clarke subdifferential possess fewer chance to be ergodic.
\end{rem}

In a recent article Polterovich \cite{Polterovich} asked general questions about the existence of invariant measures with large rotation vectors.
More precisely, his hypothesis are the existence of two Lagrangian submanifolds that have Hamiltonian rigidity but displaceable one from the other through a symplectic deformation with non trivial flux. 
Our version in the case of the cotangent bundle is related to this question but more general because we do not assume any displaceability.

\begin{prop}
 Let $a_1\neq a_2\in H^1(M)$. Suppose that  the $L_i$ are Lagrangian submanifolds isotopic to the graph of a representant of $a_i$ and $L_i$ invariant by the action of $\phi_H^1$ with  $ H|_{L_0}\leq0,H|_{L_1}\geq1$.  
Then, there exists an invariant measure $\mu$ with rotation vector $\rho(\mu)$ satisfying : 
\[\langle a_2-a_1,\rho(\mu) \rangle \geq 1\ .\]
 \end{prop}

\begin{proof}
 By \cite{MVZ}, the invariance of $L_1$ and $L_2$ implies that $\alpha_H(a_2)-\alpha_H(a_1)\geq 1$ for any Hamiltonian chord $\gamma_i$ with $\gamma(0),\gamma(1) \in L_i$.
The mean value theorem for Clarke subdifferential of Lipschitz function \cite{Clarke} gives that there exists $a_3\in H^1(M)$ barycenter of $a_1$ and $a_2$ such that \[\exists \eta\in\partial_c \alpha(a_3)\  s.t.\ \left\langle\eta,(a_2-a_1)\right\rangle\geq 1 \ .\]
We then apply corollary \ref{maincor}.
 \end{proof}

 In particular, this theorem has an application when we consider the situation of KAM theory. 
 
\begin{prop} 
 Let $H$ integrable on $T^*\mathbb{T}^n$ and $L_1$, $L_2$ be invariant KAM tori for two different Liouville class $a_1,a_2$ that survive to the perturbation of $H$ by $K$, $||K||_{C^0}\leq\epsilon$. Suppose that $H(a_1)-H(a_2)>3\epsilon$. Then there exists an invariant measure $\mu$ satisfying : \[ \left\langle \rho(\mu), a_2-a_1 \right\rangle\geq \epsilon\ .\]
\end{prop}

\section{Extension of the method to other contexts}
 
We can think about generalization of the previous work. Indeed, heuristically, we look at the behavior of $\alpha_H(0)$ when we perturb $H$ by one forms which gives result on existence of measure with prescribed rotation vector. We next consider a formal perturbation \[H_\lambda:=H+ \sum\limits_{i=1}^n \lambda_i K_i\] where $H$ and $K_i$ are autonomous Hamiltonians.

\begin{definition}
We call $E\fc \R^n\to \R$, the map defined by $E(\lambda):=\alpha_{H_\lambda}(0)$.
\end{definition}

\begin{thm}\label{local}
 Let $\eta\in \partial_c E (0)$. Then there exists an  invariant measure $m$ invariant by $\phi_H$, such that: \[\forall i, \int K_i dm=\eta_i\].
\end{thm}
 
\begin{prf}
 The proof is quite similar to the previous section. Clearly, we can consider the flow with $\lambda$ as actual variable. Thus, for each elements in the subdifferential of $\partial_c E(0)$, there exists a sequence of integer $\psi(k)$, $\lambda_{\psi(k)} \to 0$ and orbits $\gamma_{k,i} (t)= \phi_{H_{\lambda_{\psi(k)}}}^t(x_{k,i}), t\in[0,\psi(k)]$ such that:
 \[\sum\limits_{i=1}^n r_{i,k}\frac{1}{\psi(k)}\cA(\gamma_{k,i} (t))\to E(0) \ .\]
 
 \[\frac{1}{\psi(k)}\sum\limits_{i=1}^n r_{i,k} \int \sum_j  H_j\circ\gamma_{k,i}(t)dt \to \eta, k\to \infty \ .\]
 
 It remains to check that the limit of measure supported on the correct barycenter (given by the $r_i$) of the orbits $\gamma_{k,i}$ are $\phi_H^t$ invariant measures. We finish using lemma \ref{adiabatic} and mimic the proof of \ref{main_thm}.
 
\end{prf}

 We would to stress the fact that if $H_i$ are approximation of partition of unity, such a statement can be understood as a localization requirement of invariant measures.

\begin{rem}
 We can also mix the rotation vector problem and the localization of the measures.
\end{rem}
 
We finish with a conjecture:

\begin{conj}
 Let $H$ an autonomous Hamiltonian with geometrically bounded flow. Then, 
 \[E_H\fc C_c^0(TM)\to \R\]
 \[f\mapsto\alpha_{H+f}(0)-\alpha_H(0)\]
 \noindent being Lipschitz \cite{MVZ} possess a Clarke subdifferential contained in the dual of $C_c^0(T^*M)$ that we think as measures on $T^*M$. 
 For all $m\in \partial E_H(0)$, $m$ is an invariant measure of action $\alpha_H(0)$. 
\end{conj}

Because $H\mapsto\alpha_H(0)$ is an example of a symplectic partial quasi-state, we can extend this conjecture to all symplectic (partial)-quasi states.
 
 \begin{rem}
  This procedure can be performed not only for the value of $\alpha$ at $0$ but also for all other point of $H^1(M)$. 
 \end{rem}

 \section{Non autonomous case}\label{nonautonomous}

First, we need to generalize the notion of rotation vector for non-autonomous flow. 

\begin{definition}
 Let $m$ an invariant probability measure of $\phi_H$, $H\in C^2(T^*M\times\mathbb{S}^1)$. The rotation vector $\rho(m)$ is defined to be :
 
 \[ <\rho(m),a>:=\int_0^1\int_{T^*M} \pi^*\alpha(X_{H_t}(\phi^t(x))dm(x)dt=\int_{T^*M} (\int_{\gamma_x} \pi^*\alpha )dm(x)\]
 
\noindent for all $a\in H^1(M)$, $[\alpha]=a$ and where $\gamma_x$ stands for the trajectory $\{\phi^t(x),t\in [0,1]\}$.
\end{definition}

We define as well the action of an invariant measure $m$.

\begin{definition}
 Let $m$ an invariant probability measure of $\phi_H$, $H\in C^2(T^*M\times\mathbb{S}^1)$. The action $\mathcal{A}(m)$ is defined to be :
 
 \[\mathcal{A}(m):=-\int_0^1\int <\theta,X_H>(\phi^t(x))dm(x)dt+\int_0^1\int H_t(\phi^t(x)) dm(x) dt\]
\end{definition}

\begin{rem}
 \begin{itemize}
  \item If $H$ is autonomous then we recover our previous definitions.
  \item As stated in \cite{Polterovich}, the rotation vector $\rho(m)$ only depends on the time one map $\phi$ if H is compact supported. The counterexample $H(x,p)=p$ versus $H=0$ on $T^*M$ is instructive in the lights of ours theorems.
  \end{itemize}
\end{rem}

In the non-autonomous case, we cannot hope for the same degree of generality. Nevertheless, we can formulate statement about the existence of invariant measure for the time one map of periodic Hamiltonian.

\begin{thm}
Let $H\in C^2(T^*M\times\mathbb{S}^1)$ with geometrically bounded flow. Then, for all $\eta\in \partial \alpha_H(\hat\lambda)$, the exists $m$ a $\phi^1_H$ invariant measure such that $\cA(m)=\alpha_H(\hat\lambda)-\langle\eta, \hat\lambda\rangle$ and $\rho(m)=\eta$.
\end{thm}

\begin{prf}
 The proof follows essentially the same line than \ref{main_thm}. The construction is similar as limit of sum of measures supported on selected Hamiltonian paths. The method of lemma \ref{adiabatic} used for the autonomous case gives us a measure $\nu$ on $T^*M\times \mathbb{S}^1$ which is invariant by $g_t$:
 
 \[g_t\fc(x,s)\mapsto (\phi_{t+s}\phi_s^{-1},s+t) \ .\]
 
 Moreover, we have 
 
 \[\int -\langle\theta,X_{H_t}\rangle + H_t d\nu=\alpha_H(\hat\lambda)-\int \pi^*\hat\lambda(X_H)d\nu\] 
 
 \[ \int \langle\pi^*\zeta,X_H\rangle d\nu=<\zeta,\eta>\ , \forall \zeta \in H^1(M)\ .\]
 
 According to Polterovich's trick explained in appendix, there exists a $\phi$ invariant measure $m$ on $T^*M$ such that with the previous equality:

 \[\rho(m)=\eta\]
 \[\cA (m)=\alpha_H(\hat\lambda)-\rho(m)(\hat\lambda)\ .\]
 \end{prf}

\appendix
\section{Polterovich's trick from autonomous to non autonomous.}

In his recent paper Polterovich \cite{Polterovich} gives a method to correctly get an invariant probability measure of a time one flow of an non autonomous Hamiltonian on a symplectic manifold $X$ from invariant measure on $X\times\mathbb{S}^1$ satisfying good relation according to integration. More precisely, here is an adaptation to our context:

\begin{prop}\cite{Polterovich}(step 6)
   Let $g_t: T^*M\times \mathbb{S}^1\to T^*M\times \mathbb{S}^1, (x,s)\mapsto (\phi_{s+t}\phi_s^{-1}x,s+t)$ and suppose we have constructed a $g_t$ invariant measure $\nu$ on $T^*M\times \mathbb{S}^1$. Then there exists $m$, a $\phi$ invariant probability measure on $T^*M$ such that:
   
   \[\int_{T^*M\times\mathbb{S}^1} G(x,s) d\nu=\int_0^1 \int_{T^* M}G(\phi_s(x),s)dm(x)ds, \forall G\in C_c^0(T^*M\times\mathbb{S}^1)\ .\]
 
\end{prop}

\begin{proof}
 Define $A\fc (x,s)\mapsto (\phi_s(x),s)$, $B\fc (x,s) \mapsto (\phi^{-1}(x),s+1)$ and $R_t \fc (x,s)\mapsto (x,s+t)$.
 It is possible to lift $g_t$ from $T^*M\times\mathbb{S}^1$ to $T^*M\times\R$, we call this lift $\tilde g_t$. We have the following one to one correspondence:
 
 \[\left\{ \sigma,\text{probability measure on } T^*M\times\mathbb{S}^1 \right\} \longleftrightarrow \left\{ \tilde \sigma,R_1 \text{inv.  on } T^*M\times\R, \tilde \sigma(T^*M \times [0,1[)=1 \right\}\ . \]
 
\noindent A measure $\sigma$ is $g_t$ invariant if and only if $\tilde \sigma$ is $\tilde g_t$ invariant.
 
We can rewrite $\tilde g_t = AR_tA^{-1}$ and deduce that if $\tilde \sigma$ is $\tilde g_t$ invariant, then : $\tilde \sigma=A_\sharp \bar{m}$ with $\bar{m}$ a $R_t$ invariant measure which is necessarily of the form $\bar{m}=m\otimes ds$.
Moreover, $\tilde \sigma$ is $R_1$ invariant if and only if $A_\sharp \bar{m}$ is $R_1$-invariant if and only if ${A_\sharp}^{-1} {R_1}_\sharp A_\sharp \bar{m}= \bar{m}$. Thus,

\[B_\sharp(m\otimes ds)= m\otimes ds\ .\]

\noindent Then, $m$ is invariant by $\phi$. We conclude by the construction that for all $G\in C_c^0(T^*M\times \mathbb{S}^1):$

\[ \int_{T^*M\times\mathbb{S}^1} G(x,s) d\sigma=\int_0^1 \int_{T^* M}G(\phi_s(x),s)dm(x)ds, \forall G\in C_0^0(T^*M\times\mathbb{S}^1) \ .\]

\end{proof}

\end{document}